\documentclass[a4paper, 11pt]{article}
\usepackage[latin1]{inputenc}
\usepackage[T1]{fontenc}
\usepackage{amssymb}
\usepackage{amsmath}
\usepackage{amsthm}
\usepackage[english]{babel}
\usepackage{graphicx}
\usepackage{epsfig}
\usepackage{color}
\usepackage{latexsym}
\usepackage[all]{xy}
\usepackage{geometry}

\theoremstyle{plain}
\newtheorem{thm}{Theorem}
\newtheorem{theorem}{Theorem}

\newtheorem{corollary}[thm]{Corollary}

\newtheorem{lemma}{Lemma}

\theoremstyle{definition}
\newtheorem{definition}[thm]{Definition}
\newtheorem{remark}[thm]{Remark}

\title{Quillen model structures on the category of graphs}

\author{Jean-Marie Droz}

\date{\today}

\begin{document}
\maketitle

\thanks{The first author was supported by a SNF (Swiss National Fond) grant.}

\begin{abstract}
We present different ways of endowing a particular category of graphs with Quillen model structures. We show, among other things, that the core of a graph can be seen as its homotopy type in an appropriate Quillen model structure, and that an infinity of Quillen model structures exist for our particular category of graphs.
\end{abstract}

\tableofcontents

\section{Introduction}

The construction of model structures on the category of graphs is appealing for many reasons. As in algebraic topology, a model structure can be seen as a way to pass from the category of graphs to a better but coarser homotopy category. It can also help organize concepts and results of graph theory or give a firm foundation to analogies between the category of graphs and topological or algebraic categories. It may also suggest new ideas and questions about graphs. Finally, it provides many new examples of model categories and might be a tool to develop intuition and quickly check hypothesis. 

We consider Theorem \ref{coreModelStructure} and Theorem \ref{continuumOfStructures} our two main results. This last theorem determines the cardinality of the set of model structures on the category of graphs. Theorem \ref{coreModelStructure} constructs a model structure on the category of graphs, where the core of the graph is its homotopy type. This generalizes to other categories. This generalization will be explored in further work. In the present article, we propose to begin the systematic study of the model structures on the category $\mathcal{G}$ of finite undirected graphs without multiple edges. This category is often used in combinatorics \cite{nesetril} and in work applying algebraic topology to graph theoretical problems \cite{lovasz,kozlov}.

\subsection{Previous work}

While the category $\mathcal{G}$ is not a topos (see Theorem \ref{notATopos}), the category of directed graphs $\mathcal{G}_{dir}$ with multiple edges is a functor category over the category of sets and can therefore be given the structure of a topos \cite{barrWells}. The study of model structures definable over $\mathcal{G}_{dir}$ is also of great interest. Such Quillen model structures have been constructed by Bisson and Tsemo \cite{bissonTesmoHomotopical,bissonTesmoSymbolic,bissonTesmoIsospectral}.

Investigations in the category $\mathcal{G}$ with the aim of defining a notion of homotopy of graphs have also been carried out by Dochterman and Babson \cite{dochtermann}. Model structures are currently being investigated in places far removed from topology, such as set theory \cite{hassonGavrilovich} or even constructive logic \cite{awodey}.

\subsection{Plan}

 After defining an appropriate category of graphs (Section \ref{A_category_of_graphs}) and model structures (Section \ref{Model_Categories}), we look at a few simple examples of model structures on the category of graphs (Section \ref{simple_model_structures}). We then construct a model structure whose notion of homotopy type corresponds to the notion of the core of a graph (Section \ref{coreModel}). Finally, we count the model structures on the category of graphs (Section \ref{tnomcs}).

\subsection{Acknowledgments}
%\ack
This work originated from a suggestion of Paul Turner made while we were working together on a related topic and started out as a collaboration. It began at the university of Fribourg, and is now supported by an FNS grant for prospective researchers. I'd like to thank Dmitry Kozlov and Ruth Kellerhals for their support and Paul for his ideas and help. The present treatment of the core category benefited from discussions with Inna Zakharevich. I am grateful to her and Emanuele Delucchi for helping me proofread the present article.

\section{A category of graphs}
\label{A_category_of_graphs}
\subsection{Definitions}
We will work with undirected finite graphs that may have loops but have at most one edge between any two (not necessarily different) vertices. We therefore define graphs in the following way:
\begin{definition}
A {\it graph} $G$ is a symmetric binary relation on a finite set. We write $G=(V_G,E_G)$ where $V_G$ is the underlying set of {\it vertices} and $E_G$ is a set of unordered pairs of vertices called edges.  
\end{definition} 
\begin{definition}
A {\it homomorphism} $f$ between the graphs $G$ and $H$ is a map $f:V_G \to V_H$ such that $\forall x,y\in V_G,\, (x,y)\in E_G \Rightarrow (f(x),f(y))\in E_H$.
\end{definition}

Figure \ref{fig:characterisations} contains examples of graphs and homomorphisms. 

\begin{definition}
We use $\mathcal{G}_{all}$ to denote the category of graphs and graph homomorphisms. We construct the category $\mathcal{G}$ as the full subcategory of $\mathcal{G}_{all}$ with set of objects obtained by choosing one representative for each isomorphism type of graph.
\end{definition}
We observe that the categories $\mathcal{G}_{all}$ and $\mathcal{G}$ are equivalent and that all the constructions we use behave well with respect to replacing a category with an equivalent one. Our choice of working mainly with $\mathcal{G}$ is thus a matter of taste and convenience, since  $\mathcal{G}$ is small.
\subsection{Basic properties of the category of graphs}
\begin{definition}
A category is {\it finitely complete} if it possesses all finite limits. It is {\it finitely cocomplete} if it possesses all finite colimits.
\end{definition}
\begin{theorem}[Folklore]
The category of graphs $\mathcal{G}$ is finitely complete and finitely cocomplete.
\end{theorem}
\begin{proof}
A category with finite products and equalizers is finitely complete. Dually, a category with coproducts and coequalizers is finitely cocomplete \cite{barrWells}. (The existence of initial and terminal objects is a consequence of the existence of coproducts and products over the empty set.) All the necessary constructions of limits and colimits are obtained by simple modifications of their analogues in the category of sets.
Let $H$ and $K$ be two graphs. The coproduct $H+K$ is the disjoint union of the two graphs. The product $H\times K$ is  
\[ 
(V_H\times V_K, \{((h_1,k_1),(h_2,k_2))\in (V_H\times V_K)^2\mid (h_1,h_2)\in E_H \wedge (k_1,k_2)\in E_K\}).
\]
For two morphisms $f,g:H\rightarrow K$, the equalizer is the subgraph of $H$ induced by the vertices sent to the same vertex by $f$ and $g$. The coequalizer of $f$ and $g$ is the quotient of $K$ by the equivalence relation generated by the set of pairs of vertices $\{(f(x),g(x))\mid x\in V_H\}$. 
\end{proof}

Our category of graphs is of interest for combinatorics, because it permits the expression of graph theoretical properties and problems in categorical language \cite{nesetril}. For example, a graph is $k$-colorable (see \ref{chromatic}) exactly if there exists a morphism from it to the complete graph without loops on $k$ vertices. The following theorem shows that, in a sense, any assertion about finite graphs can be translated into an assertion about our category $\mathcal{G}$.

\begin{theorem}
\label{noAutomorphism}
The category $\mathcal{G}$ has no non-trivial automorphism.
\end{theorem}
 \begin{proof}
 The proof consists in showing that everything about graphs can be expressed in terms of morphisms in $\mathcal{G}$.
 \begin{itemize}
\item{The empty graph is characterized as the initial element of $\mathcal{G}$.}
\item{The graph $P$ with only one vertex and no edge is the only graph that maps to all others except the empty graph.}
\item{The set of vertices of an object in $\mathcal{G}$ is given by the set of morphisms to it from $P$.}
\item{The graph $T$ with one vertex and one edge (a loop) is characterized by being the terminal object.}
\item{Graphs without loops are characterized by the non-existence of morphism from $T$ to them.}
\item{The property of having two vertices is expressed by having exactly two morphisms from $P$.}
\item{The graph $E$ with two vertices and one edge between them is fixed by any automorphism (of the category). Among graphs without loop, it is the unique graph with two vertices that does not map to any other graph with two vertices and without loop. }
\item{The complete structure of an arbitrary graph $G$ can be reconstructed from the sets of maps $hom(E, G)$, $hom(P,G)$ and which maps from $hom(P,G)$ factorize through which maps $hom(E, G)$. The set $hom(E, G)$ represents the edges of $G$. The set $hom(P,G)$ represents the vertices of $G$. A map $m\in hom(P, G)$ factorizes through a map $n\in hom(E, G)$, exactly when the corresponding vertices and edge are incident.}
\end{itemize}
 \end{proof}
\begin{theorem}
\label{notATopos}
The category $\mathcal{G}$ cannot be given the structure of a topos.
\end{theorem}
 \begin{proof}
 Let $Sub: \mathcal{G}\rightarrow \mathcal{SET}$ be the {\it subobject functor} of the category of graphs. One of the definitions of a topos states \cite{barrWells} that there exists a natural isomorphism \allowbreak $Hom(B,P(A))\leftrightarrow Sub(B\times A)$, for some functor $P: \mathcal{G}\rightarrow \mathcal{G}$. (In the topos $\mathcal{SET}$, $P$ is the powerset functor.) Let $A$ be the complete graph (without loops) on three vertices, $B$ the graph with a mere vertex and $B'$ the graph with one vertex and one loop. For any functor $P$, we have the inequality $|Hom(B,P(A))|\geq |Hom(B',P(A))|$. We compute that $8=|Sub(B\times A)|<|Sub(B'\times A)|=18$. The theorem follows by the contradiction between the two inequalities and the existence of a natural isomorphism.  
 \end{proof}
 By contrast, it is asserted in \cite[Section 8]{vigna} that our category of graphs is a quasitopos. This article also explains that other natural categories of graphs are topoi.  

\section{Model categories}
\label{Model_Categories}
\subsection{Categorical preliminaries}
\begin{definition}
In a category $\mathcal{C}$, we say that the morphism $f:A\to B$ {\it lifts on the left} of the morphism $g:C\to D$ if for any commutative diagram of solid arrows:
\[
\xymatrix{
A \ar[r]\ar[d]_{f} & C \ar[d]^{g} \\
B \ar[r] \ar@{-->}[ur]^{h} & D 
}
\]
there is a map $h$ which makes the complete diagram commutative.

\end{definition}
We will often speak of a map {\it lifting on the right} of another map, or of a commutative square having the {\it lifting property} with the obvious definitions.

Note that $h$ is not required to be unique. As examples of this definition, we give some characterizations of a few properties of maps between graphs in terms of lifting properties. Those characterizations and others of the same kind will often be used below.

\begin{definition}
Let $E$ be the graph with two vertices and one edge between them. A morphism $f:G\rightarrow H$ is said to be edge-surjective, if all morphisms from $E$ to $H$ factorize through $f$. 
\end{definition}

\begin{figure}
\centering
\includegraphics[ width=12cm]{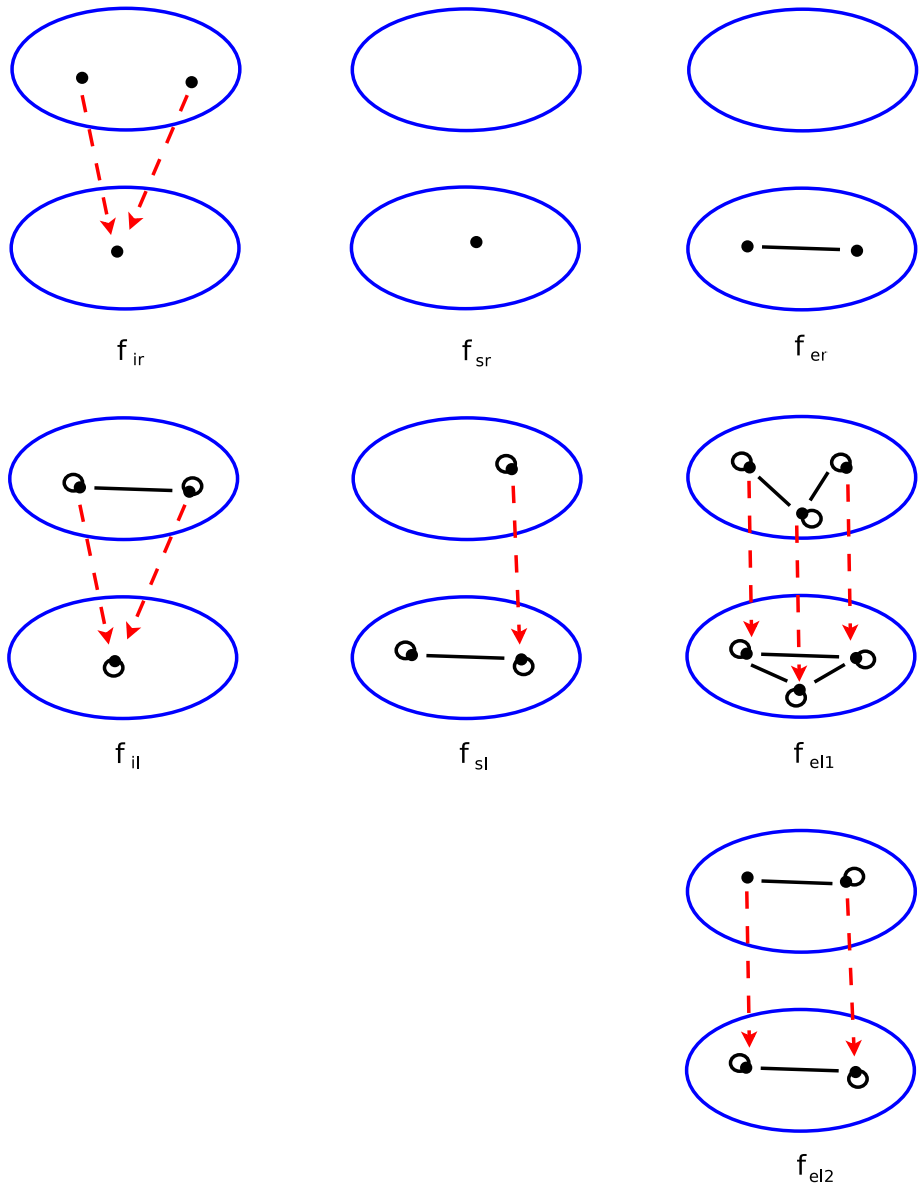}
\caption{Morphisms between graphs are depicted by two ellipses containing the domain and the codomain (the domain is in the upper ellipse, the codomain in the lower ellipse), and arrows determining where the vertices of the domains are mapped.}
\label{fig:characterisations}
\end{figure}

\begin{theorem}
\label{characterisations}
A morphism of $\mathcal{G}$ is injective exactly if it lifts on the right of $f_{ir}$, exactly if it lifts on the left of $f_{il}$. It is surjective exactly if it lifts on the right of $f_{sr}$, exactly if it lifts on the left of $f_{sl}$. It is edge-surjective exactly if it lifts on the right of $f_{er}$, exactly if it lifts on the left of $f_{el1}$ and $f_{el2}$. The graphs $f_{ir}$, $f_{il}$, $f_{sr}$, $f_{sl}$, $f_{el}$ and $f_{er}$ are defined by Figure \ref{fig:characterisations}. More precisely, concerning edge-surjectivity, a morphism has at least one vertex with a loop in the preimage of every vertex with a loop only if it lifts on the left of $f_{el2}$ and it is edge-surjective for non-loops exactly if it lifts on the left of $f_{el1}$.
\end{theorem}
\begin{proof}[Sketch of proof]
The simple proof of this theorem requires the examination of a long sequence of similar constructions, either for constructing the appropriate lifting map or for producing commutative squares without the lifting property. We only prove, as example of the method, that a morphism is injective if and only if it lifts on the left of $f_{il}$. For the ``if'' part, we need to build a commutative square from a non-injective morphism $f:F_1\rightarrow F_2$ to $f_{il}:G_1\rightarrow G_2$ without the lifting property. Any morphism $u:F_1\rightarrow G_1$ sending two vertices with the same image under $f$ to different vertices can be completed uniquely in such a commutative square (since $G_2$ is the terminal object). To prove the ``only if'' part, we must construct for a diagram of solid arrows with $f$ injective, an arrow $h$ which makes the diagram commute. 
\[
\xymatrix{
F_1 \ar[r]^{u}\ar[d]_{f} & G_1 \ar[d]^{f_{il}} \\
F_2 \ar[r] \ar@{-->}[ur]^{h} & G_2 
}
\]
This is accomplished by having $h$ send every point $x$ of $F_2$ either to $u\circ f^{-1} (x)$ if $x$ has a preimage or to an arbitrary point of $G_1$ otherwise.
\end{proof}

\begin{definition}
A morphism $r:A\rightarrow B$ in a category is called a {\it retraction} if it is possible to factorize the identity of $B$ as $1_B=r\circ s$ for some map $s$. We then call $B$ a retract of $A$. Dually, a morphism $s:A\rightarrow B$ is called a {\it section} if it is possible to factorize the identity of $A$ as $1_A=r\circ s$ for some map $r$. 
\end{definition}

We also often use the word ``retraction'' in a different way.

\begin{definition}
For morphisms $f$ and $g$ of $\mathcal{C}$, we say that $f$ is a {\it retraction} of $g$, if it is a retract of $g$ in the category of morphisms of $\mathcal{C}$, in other words, if there are maps of dotted arrows that make the following diagram commute.
\[
\xymatrix{
A \ar@{.>}[r]\ar[d]_{f}\ar@/^1pc/[rr]^{1} & B \ar@{.>}[r]\ar[d]^{g} & A\ar[d]_{f}\\
A' \ar@{.>}[r] \ar@/_1pc/[rr]_{1}  & B'\ar@{.>}[r] & A'
}
\]
\end{definition}

The following definitions will enable us, among other things, to give a compact reformulation of the axioms of a  model structure.

\begin{definition}
A {\it maximal lifting system} $(A,B)$ in a category $\mathcal{C}$ is a pair of sets of morphisms, such that $A$ is the set of all morphisms lifting on the left of $B$ and $B$ is the set of all morphisms lifting on the right of $A$.
\end{definition}

\begin{theorem}[Folklore]
\label{wfs}
If $(A,B)$ is a maximal lifting system, then $A$ and $B$ contain all identity morphisms and are closed under composition and under taking retraction. Moreover, $A$ is closed under coproduct and under taking the pushout along any morphism and $B$ under product and under taking the pullback along any morphism.
\end{theorem}

\begin{definition}
\label{wfsDef}
A {\it weak factorization system} $(A,B)$ in the category $\mathcal{C}$ is a maximal lifting system such that any morphism in $\mathcal{C}$, is factorisable as $g \circ f$ with $f\in A$ and $g\in B$.
\end{definition} 

\subsection{Definition of model structures}

Since we work in the small category $\mathcal{G}$, we give a definition of model structures appropriate for small categories and, in the rest of the article, avoid the careful distinction that must often be made between sets and classes. It is natural to ask for all small limits and colimits to exist in a category which might be a class, and analogously, in a category which is a set, to ask only for finite limits and colimits. (Moreover, if a small category had all small limits or colimits, it would be quite peculiar, having at most one morphism between any two objects \cite[p.110]{maclane}.)
\begin{definition}
\label{originalAxioms}
A {\it model structure} $\mathbb{M}$ on a finitely complete and finitely cocomplete category $\mathcal{C}$ is the specification of three subcategories of $\mathcal{C}$ on the same object set called the category of {\it weak equivalences} ($\mathcal{M}_{we}$), the category of {\it cofibrations} ($\mathcal{M}_{cof}$) and the category of {\it fibrations} ($\mathcal{M}_{fib}$). Those three subcategories should respect the axioms given below.
\begin{itemize}
\item{\textbf{ Two-of-three axiom } For composable morphisms $f$ and $g$, if two of the maps $f$, $g$ and $f\circ g$ are weak equivalences, then so is the third.}
\item{\textbf{ Lifting axiom} For a cofibration $f$ and a fibration $g$, if one of the two is a weak equivalence, $f$ lifts on the right of $g$.}
\item{\textbf{ Factorization axiom} Any morphism $m\in \mathcal{M}$ can be written $m=f\circ g$ for a cofibration $g$ and a fibration $f$. Moreover, either $f$ or $g$ can be chosen to be a weak equivalence.}
\item{\textbf{ Retraction axiom} The subcategories $\mathcal{M}_{weak}$, $\mathcal{M}_{cof}$ and $\mathcal{M}_{fib}$ are closed under taking retractions.}
\end{itemize}
We call {\it model category}, a category with a model structure on it.
\end{definition}
The modern definitions of a model category usually require the category to have all small limits and colimits and not just finite limits and colimits. However, Quillen's original definition of a ``closed model category'' only asks for finite limits and colimits and so all classical properties of model categories are true in our ``finite'' setting. We also checked that the other references we give for our theorems about model categories \cite{dwyerSpalinsky,joyal} only use finite limits and finite colimits in the relevant proofs.

The axioms are invariant by replacing the category $\mathcal{C}$ by its dual $\mathcal{C}^{op}$ and exchanging the role of fibrations and cofibrations. This is often useful in proofs of general facts about model structures like Theorem \ref{invar} or \ref{riehl}. 

Verification of the classical axioms above is often difficult, a more practical and equivalent set of axioms is given below.

We state a few basic properties of the subcategories given with a model structure. They are direct consequences of Theorem \ref{wfs}.

\begin{theorem}[Folklore {\cite[Section 3]{dwyerSpalinsky}}]
\label{invar}
The pushout along a morphism of a cofibration is a cofibration. The pushout along a morphism of an acyclic cofibration is an acyclic cofibration. By duality, the pullback along any morphism of a fibration or an acyclic fibration is a fibration or an acyclic fibration. 
\end{theorem} 

We now provide a few results that will help streamline the proofs that triples of sets of maps called weak equivalences, fibrations and cofibrations form model structures. We often use the designations {\it acyclic cofibrations} or {\it acyclic fibrations} to speak of the maps that are weak equivalences and cofibrations or weak equivalences and fibrations.

Using the first part of Theorem \ref{wfs}, we obtain:

\begin{corollary}[Folklore {\cite[Section 3]{dwyerSpalinsky}}]
\label{riehl}
If the sets of fibrations, cofibrations, acyclic fibration and acyclic cofibrations are each maximal among sets satisfying the appropriate lifting properties, then they all contain the identity morphism and are closed by composition. Moreover, the axiom of retraction for each kind of map follows. 
\end{corollary} 

We need a useful but little known theorem.

\begin{theorem}[M.~Tierney \cite{joyal}]
\label{Tierney}
The axiom of retraction for weak equivalences follows from the others.
\end{theorem}

We deduce the following alternative definition of a model structure.
\begin{definition}
\label{shorterAxioms}
A {\it model structure} $\mathbb{M}$ on a finitely complete and finitely cocomplete category $\mathcal{C}$ is the specification of three classes of morphisms in $\mathcal{C}$ called the {\it weak equivalences} ($\mathcal{M}_{we}$), the {\it cofibrations} ($\mathcal{M}_{cof}$) and the {\it fibrations} ($\mathcal{M}_{fib}$). Those three sets should respect the following axioms.
\begin{itemize}
\item{\textbf{ Weak factorization system axiom } The following pairs are weak factorization systems (see Definition \ref{wfsDef}): $(\mathcal{M}_{cof},\mathcal{M}_{fib}\cap \mathcal{M}_{we})$, $(\mathcal{M}_{cof}\cap \mathcal{M}_{we},\mathcal{M}_{fib})$}
\item{\textbf{ Two-of-three axiom } For morphisms $f$ and $g$, if two of the maps $f$, $g$ and $f\circ g$ are weak equivalences, then so is the third.}
\end{itemize}
\end{definition}
\subsection{Examples}

In any finitely complete and finitely cocomplete category, taking as weak equivalence all the isomorphisms, and as cofibrations and fibrations all morphisms, we obtain a model structure on that category. We call this model structure the trivial model structure. In the category of sets, two simple model categories arise by taking as weak equivalences the isomorphisms and either taking as cofibrations the injective morphisms and as fibrations the surjective morphisms or (somewhat surprisingly) the reverse.

Classic examples of model structures arise in algebraic topology over categories of topological spaces, taking as weak equivalences either homotopies or maps inducing isomorphisms on the homotopy groups, see for example \cite{mayPonto} for details. Model categories were discovered by Quillen in this context \cite{quillen}. 

\subsection{The homotopy category of a model category}

Model structures on a category $\mathcal{C}$ are mainly used to understand the relationship between the category $\mathcal{C}$ and one of its localizations called the homotopy category of $\mathcal{C}$. If $\mathcal{C}$ is equipped with a model structure $\mathbb{M}=(\mathcal{M}_{cof},\mathcal{M}_{fib}\cap \mathcal{M}_{we})$, its homotopy category (with respect to the model structure $\mathbb{M}$) is written ${\mathcal{M}_{we}}^{-1}\mathcal{C}$. This notation comes from the following theorem.
\begin{theorem}[Quillen \cite{quillen}]
\label{homotopyAsLocalization}
The homotopy category of a model category is its localization at the weak equivalences.
\end{theorem}
There is a canonical functor from $\mathcal{C}$ to ${\mathcal{M}_{we}}^{-1}\mathcal{C}$ which associates to every object of $\mathcal{C}$ an object of ${\mathcal{M}_{we}}^{-1}\mathcal{C}$ called its {\it homotopy type}.
Although the homotopy category only depends on the weak equivalences, the rest of the model structure is crucial for controlling the relationship between the category and its localization and for many constructions inside the homotopy category.
The classic example of a homotopy category is that of the category of topological spaces, where weak equivalences are continuous maps inducing isomorphisms on homotopy groups. In this case, Whitehead's theorem claims that the homotopy types can be represented by $CW$-complexes. See again \cite{mayPonto} or \cite{dwyerSpalinsky}.

%%%%%%%%%%%%%%%%%%%%%%%%%%%%%%%%%%%%%%%%%%%%%%%%%%%%%%%%%%%%%%%%%%%%%%%%%%%%%%%%%%%%%%%%%%%%%%%%%%%%%%%%%%%%%%%%%%%%%%%%%%%%%%%%%%%%%%%%%%%%%%%%%%%%%%%%%% 
\section{Model structures on the category of graphs}
\label{mcotcg}

\subsection{Simple model structures on the category of graphs}
\label{simple_model_structures}
We begin by describing a few trivial model structures on the category of graphs. 
\begin{theorem}
There are three trivial model structures on the category of graphs, obtained by choosing the subcategories of fibration, cofibrations and weak equivalences to be either the whole category $\mathcal{G}$ or $\mathcal{G}_{iso}$, its restriction to isomorphisms. These model structures are:
\begin{itemize}
\item $\mathcal{M}_{we}=\mathcal{G}$, $\mathcal{M}_{cof}=\mathcal{G}$ and $\mathcal{M}_{fib}=\mathcal{G}_{iso}$
\item $\mathcal{M}_{we}=\mathcal{G}$, $\mathcal{M}_{cof}=\mathcal{G}_{iso}$ and $\mathcal{M}_{fib}=\mathcal{G}$
\item $\mathcal{M}_{we}=\mathcal{G}_{iso}$, $\mathcal{M}_{cof}=\mathcal{G}$ and $\mathcal{M}_{fib}=\mathcal{G}$
\end{itemize} 
\end{theorem}

As a first nontrivial example, we define a model structure for which the homotopy type of a graph is given by its set of connected components.
\begin{theorem}
Let three subcategories of $\mathcal{G}$ called weak equivalences, cofibration and fibrations, be defined as follows.
\begin{itemize}
\item The morphisms inducing isomorphisms between the sets of connected components of the two graphs are the weak equivalences $\mathcal{CC}_{we}$.
\item The trivial subcategory of all morphisms is the subcategory of cofibrations $\mathcal{CC}_{cof}$.
\item The fibrations $\mathcal{CC}_{fib}$ are the maps satisfying the right lifting properties with respect to the acyclic cofibrations $\mathcal{CC}_{cof}$.
\end{itemize} 
The triple of subcategories $(\mathcal{CC}_{we},\mathcal{CC}_{cof},\mathcal{CC}_{fib})$ forms a model structure $\mathbb{CC}$ on $\mathcal{G}$. Its homotopy category is (isomorphic to) the category of finite sets. According to this model structure, the homotopy type of a graph is its set of connected components.
\end{theorem}
\begin{proof}
We begin by checking that the axioms in Definition \ref{shorterAxioms} are satisfied. The only non-obvious part is that the acyclic cofibrations and the fibrations form a weak factorization system. 
\begin{figure}
\centering
\includegraphics[ width=12cm ]{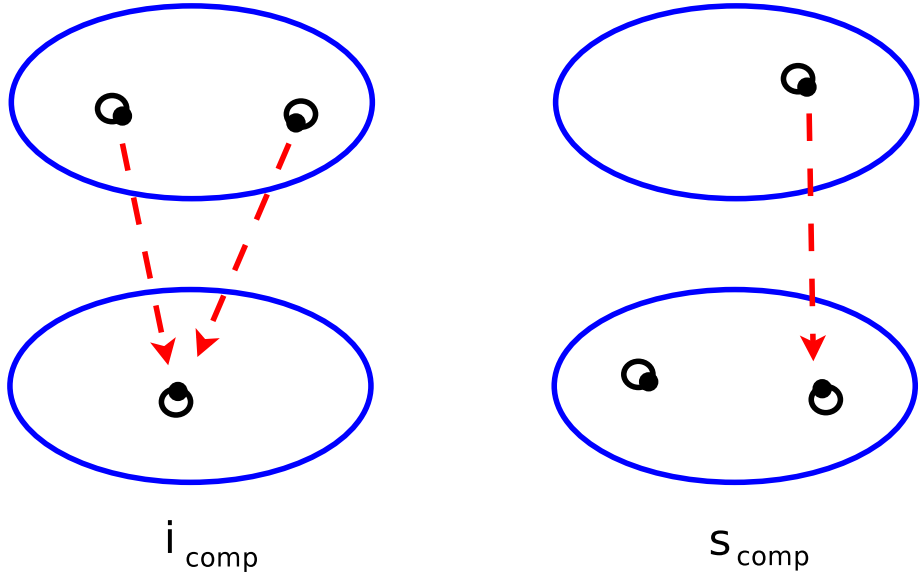}
\caption{The two morphisms determining the acyclic cofibrations by their lifting property.}
\label{fig:connectedCompFib}
\end{figure}
That the acyclic cofibrations\footnote{By a slight abuse of language, that will often be repeated, before proving that we have a model structure, we use the term acyclic (co)fibrations to denote maps that are both weak equivalences and (co)fibrations.} are exactly the morphisms lifting on the left of the fibrations is a consequence of the following fact. They are exactly the morphisms lifting on the left of the morphisms $s_{comp}$ and $i_{comp}$ of Figure \ref{fig:connectedCompFib}. By definition, the fibrations have the appropriate lifting property.

To prove that any morphism can be factorized in an acyclic cofibration and a fibration, we begin by characterizing the fibrations.
A morphism  $f:G\rightarrow H$ lifts on the right of all acyclic cofibrations exactly if the inverse image of any connected component $H_i\subset H$ is a (possibly empty) disjoint union of copies of $H_i$. Let $f:G\rightarrow H$ be any morphism of $\mathcal{G}$. Let $f_i$ be the restrictions of $f$ to the $i$-th connected components of $G$ and the component in $H$ containing its image. We define $f_{acof}$ as the coproduct of the $f_i$'s. (The coproduct in the category of morphisms is the ``disjoint union'' of the morphisms in the intuitive way.) Clearly, $f_{acof}$ is an acyclic cofibration. We then see that $f$ can be factorized as $f= f'_{fib}\circ f_{fib}\circ f_{acof}$, where $f_{fib}$ is the morphism identifying the different copies of the same component of $H$ that are in the codomain of $f_{acof}$, and $f'_{fib}$ is the injection which ``adds'' the connected components of $H$ with empty preimage by $f$. This provides the necessary factorization, since $f'_{fib}\circ f_{fib}$ is a fibration.   

The assertions about the homotopy category and homotopy types are straightforward consequences of the definition of the weak equivalences and Theorem \ref{homotopyAsLocalization}.
\end{proof}

The last model structure only had isomorphisms as acyclic fibrations. We construct a simple model structure where the acyclic cofibrations are the isomorphisms. Its homotopy category is the category of finite graphs in which every vertex is incident to at least one edge.

\begin{definition}
If each vertex of a graph is adjacent to at least one edge, we call this graph {\it furbished}. The {\it furbished part} of a graph is its greatest furbished induced subgraph.
\end{definition}

\begin{theorem}
Let three sets of maps be defined as follows.
\begin{itemize}
\item The morphisms inducing isomorphisms on the furbished part are the weak equivalences $\mathcal{E}_{we}$.
\item The trivial subcategory of all morphisms is the subcategory of fibrations $\mathcal{E}_{fib}$.
\item The cofibrations $\mathcal{E}_{cof}$ are the maps lifting left of the acyclic fibrations $\mathcal{E}_{fib}$.
\end{itemize} 
The triple of subcategories $(\mathcal{E}_{we},\mathcal{E}_{cof},\mathcal{E}_{fib})$ forms a model structure $\mathbb{E}$. Its homotopy category is (isomorphic to) the category of furbished graphs and the homotopy type of a graph is its furbished part.
\end{theorem}
\begin{proof} 
We check that the axioms in Definition \ref{shorterAxioms} are satisfied. The only non-obvious part is that the cofibrations and the acyclic fibrations form a weak factorization system. 
\begin{figure}
\centering
\includegraphics[width=12cm ]{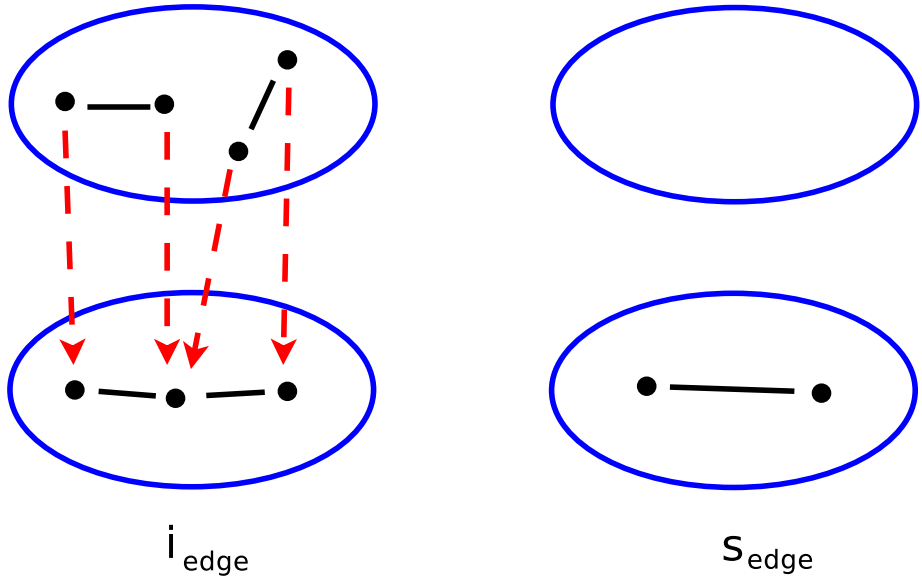}
\caption{The two morphisms determining the acyclic fibrations by their lifting property.}
\label{fig:isolatedPointsCof}
\end{figure}
That the acyclic fibrations are exactly the morphisms lifting on the right of the cofibrations is a consequence of the following. They are exactly the morphisms lifting on the right of the morphisms $i_{edge}$ and $s_{edge}$ of Figure \ref{fig:isolatedPointsCof}. By definition, the cofibrations have the appropriate lifting property. It remains to prove the factorizability of morphisms in a cofibration and an acyclic fibration.

Since the cofibrations and the acyclic fibrations form a lifting system, by Theorem \ref{wfs}, the cofibrations are closed under taking coproducts, pushouts and compositions. We deduce that any surjective morphism between furbished graphs are cofibrations. This comes from the fact that all such morphisms are compositions of surjective morphisms in which at most one vertex in the codomain has two preimages and those last morphisms are pushouts of $i_{edge}$. Furthermore, any injection between furbished graphs is a composition of a sequence of morphisms which are coproducts of an isomorphism and $s_{edge}$, and a surjection. We conclude that any morphism between furbished graphs is a cofibration. 
Any morphism between graphs induces by restriction a morphism between their maximal furbished subgraphs. The morphisms $i_{edge}$ and $s_{edge}$ lift on the left of any morphism inducing isomorphism on the furbished part. Consequently, all such morphisms are acyclic fibrations. 

Any morphism is a composition of the coproduct of an isomorphism and a morphism between furbished graphs, and a morphism inducing isomorphism on the furbished part. In other words, every morphism can be factorized in a cofibration and an acyclic fibration. 

The assertions about the homotopy category and homotopy types are easy consequences of the definition of the weak equivalences and Theorem \ref{homotopyAsLocalization}.
\end{proof}

\subsection{The core model structure}
\label{coreModel}

\begin{definition}
The {\it core} of a graph $G$ is its smallest retract $G_{core}$. A graph {\it is a core} if it has no smaller retract.
\end{definition}
\begin{figure}
\centering
\includegraphics{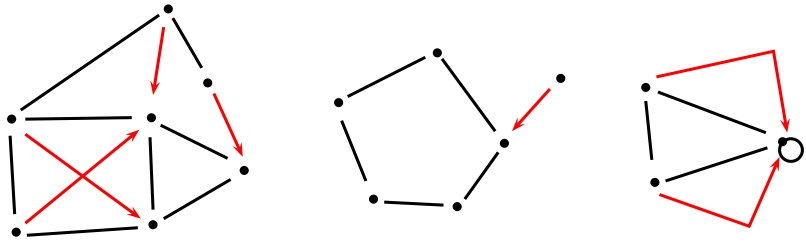}
\caption{Examples of graphs retracting to their cores.}
\end{figure}
\begin{lemma}[Folklore]
\label{powerOfEndo}
Every endomorphism of a finite set has a power which is idempotent.
\end{lemma}
\begin{proof}
Since an endomorphism $f$ acts as a permutation on its eventual image, if $f^n$, for some $n>0$, sends all elements of its domain to the eventual image $I$ of $f$, $f^{n\cdot|I|!}$ is an idemptotent. 
\end{proof}
This is a nice generalization (and a consequence) of the fact that in a finite group, any element has a power which is the identity.
We notice that in the category of sets, among endomorphisms, the idempotents are the retractions.

\begin{theorem}[Hell, Ne\v{s}et\v{r}il {\cite[Section 2.8]{nesetril}}]
A graph is a core if and only if it has no homomorphism to a proper subgraph.
\end{theorem}
\begin{proof}
This is a direct consequence of Lemma \ref{powerOfEndo}.
\end{proof}
\begin{corollary}
An endomorphism of a core is an automorphism.
\end{corollary}
\begin{theorem}[Hell, Ne\v{s}et\v{r}il {\cite[Section 2.8]{nesetril}}]
The core of a graph is unique up to isomorphism.
\end{theorem}
\begin{proof}
If a graph $G$ retracts to both $H_1$ and $H_2$ that are cores, the retractions induce maps in both directions between $H_1$ and $H_2$. The compositions of those maps must be automorphisms and thus both maps are both surjective and injective. We deduce that both maps are isomorphisms and thus the core is unique up to isomorphism.
\end{proof}

Considered as a subgraph, the core of a graph needs not be unique, although its isomorphism type is. Any graph $G$ is equipped with sections from its core and retractions to its core. A map between graphs $f:G\rightarrow H$ induces maps $f_{core}:G_{core}\rightarrow H_{core}$ by composition. The induced map is unique up to automorphisms of its range and of its domain.
\[
\xymatrix{
G_{core} \ar[r]\ar@/^1pc/[rrr]^{f_{core}} & G \ar[r]_{f} & H \ar[r] & H_{core}}
\]  
\begin{corollary}
\label{isoCore}
A map $f:G\rightarrow H$ induces isomorphism on the cores if and only if $G$ and $H$ have isomorphic cores.
\end{corollary}
\begin{corollary}
\label{mapCore}
A map $f:G\rightarrow H$ induces isomorphism on the cores if and only if there is a map from $H$ to $G$.
\end{corollary}

\begin{definition}
\label{quasiOrder}
 We define a partial order relation on the objects of $\mathcal{G}$ by setting, for graphs $G$ and $H$, $G\geq H \Leftrightarrow \exists h\in Hom(G,H)$. We call $\mathcal{G}_{core}$ the poset (seen as a category) obtained by restricting the order relation to cores. On the whole category $\mathcal{G}$, the order relation is only a quasi-order, since different graphs can have the same core and will then be equivalent for the ordering.
\end{definition}

Two graphs with homomorphisms in both directions between them are equivalent in the quasi-order and any morphism between them will induce an isomorphism on their cores. Conversely, all homomorphisms of graphs inducing isomorphisms on the cores arise in this way.

\begin{theorem}
\label{coreModelStructure}
There is one model structure whose weak equivalences are the maps inducing isomorphisms on the cores and whose cofibrations are the canonical injections in a coproduct. Its acyclic fibrations are the retractions.
\end{theorem}

To prove Theorem \ref{coreModelStructure}, we need a little categorical lemma.
\begin{lemma}
\label{initialTerminalLifting}
The retractions are exactly the maps lifting on the right of morphisms with the initial object of the category as domain. Dually, the sections are the maps lifting on the left of the maps with terminal object as codomain.
\end{lemma}
\begin{proof}
Let $f:A\rightarrow B$ be a morphism lifting on the left of the morphism with domain the initial object and $B$ as codomain.
We obtain a commutative square where $I$ denotes the initial object:
\[\xymatrix{
I \ar[r]\ar[d] & A \ar[d]^{f} \\
B \ar[r]_{1_B} \ar@{-->}[ur]^{h} & B 
}\]
Thus $1_B=f\circ h$ and $f$ is a retraction.
If $r:A\rightarrow B$ is a retraction with section $s$,
for any commutative diagram of solid arrows: 
\[\xymatrix{
I \ar[r]\ar[d] & A \ar[d]^{r} \\
C \ar[r]_{g} \ar@{-->}[ur]^{h} & B 
}\]
We obtain a diagonal map $h=s\circ g$ that makes the whole diagram commute. This means that a retraction lifts on the right of any map with initial domain.
The rest of the lemma follows by duality. 
\end{proof}

\begin{proof}[Proof of Theorem \ref{coreModelStructure}]
The cofibrations and acyclic cofibrations of the model structures are given. Therefore, the fibrations are exactly the maps with the right lifting property with respect to the acyclic cofibrations and the model structure in question is uniquely determined.

\begin{figure}
\centering
\includegraphics[ width=12cm ]{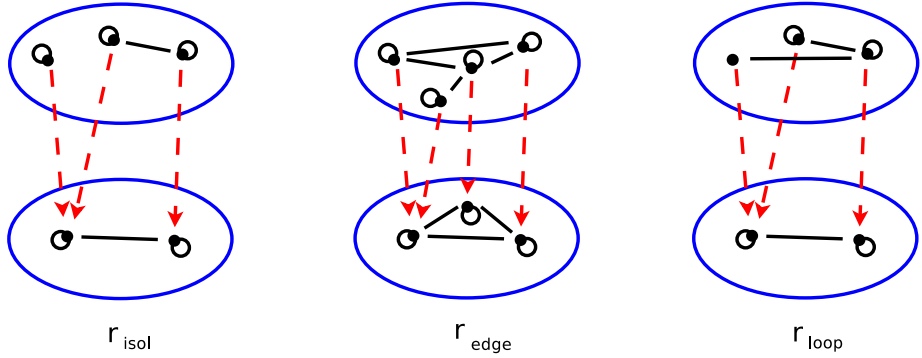}
\caption{Three retractions participating in the definition of the cofibrations by their lifting properties.}
\label{coreDefining}
\end{figure}

Using Definition \ref{shorterAxioms}, we show that we have a model structure. We proceed in four steps. We first prove that the (cofibrations, acyclic fibrations) pair and the (acyclic cofibrations, fibrations) pair are maximal lifting systems, while gaining information about the subcategories. Second,  we check that the maps lifting on the right of the cofibrations really are the fibrations which are weak equivalences (The dual will be obvious). Then, we prove that the maximal lifting systems are actually weak factorization systems. We finish the proof by showing that the Two-of-three axiom holds.

\noindent
{\bf Step 1.} We begin by proving that the set of cofibrations and the set of retractions form a maximal lifting system, and thus that we must choose the retractions as acyclic fibrations. By Lemma \ref{initialTerminalLifting}, there is a maximal lifting system $(A,B)$ where $A$ contains all the maps with empty domain and $B$ all the retractions. Again by Lemma \ref{initialTerminalLifting}, we know that $B$ contains exactly the retractions.
The morphisms $r_{isol}$, $r_{edge}$ and $r_{loop}$ of Figure \ref{coreDefining} are retractions. So is the morphism $f_{il}$ of Figure \ref{fig:characterisations}. The morphisms lifting on the left of $f_{il}$ are injections by Theorem \ref{characterisations}. The morphisms lifting on the left  of $r_{edge}$ and $r_{loop}$ are morphism where no edge is added between two points that have preimages (The map $r_{edge}$ forbids the new edge to be between different points, the map $r_{loop}$ forbids new loops). We observe that in any map $f$ lifting on the left of $r_{isol}$, a vertex of the image $Im(f)$ cannot be adjacent to a vertex in the complement of the image. We conclude that any map in $A$ is a canonical injection in a coproduct. By Theorem \ref{wfs}, $A$ must contain all isomorphisms and therefore all coproducts of an isomorphism with a map with an empty domain, that is all canonical injections in a coproduct. So the cofibration and retractions form a maximal lifting system, since they are equal to the sets $A$ and $B$.

We now prove that the other pair consisting of the acyclic cofibrations and the fibrations is a maximal lifting system. The acyclic cofibrations are canonical injections in a coproduct and induce isomorphisms on the core. If a map $f:A\rightarrow A+B$ is an acyclic cofibration, there is a map $h:B\rightarrow A$ which is obtained by composition of the retraction of $B$ to the core of $A+B$ and the inclusion of the (identical) core of $A$ in $A$. The existence of this map characterizes acyclic cofibrations among the cofibrations. We notice in particular that the acyclic cofibrations are the cofibrations that are sections. We simply define the fibrations to be the maps lifting on the right of the acyclic cofibrations. Since our set of fibrations contains all maps lifting on the right of a subset of the cofibrations, it also contains the acyclic fibrations. Therefore, all the maps that lift at the left of the fibrations are canonical injections in coproducts. Since the acyclic cofibrations are sections, the fibrations contain all maps with terminal codomain, and therefore all maps lifting at their left are sections by Lemma \ref{initialTerminalLifting}. We deduce that the acyclic cofibrations and the fibrations form a maximal lifting system.

\noindent
{\bf Step 2.} We need to show that the weak equivalences lifting on the right of the acyclic cofibrations are the same as the maps lifting on the right of the cofibrations. Clearly, the retractions lift on the right of the acyclic cofibrations (actually of all cofibrations) and are weak equivalences. Therefore, we only need to demonstrate the converse. Let the weak equivalence $f:G\rightarrow H$ lift on the right of the acyclic cofibrations. Since it lifts on the right of the canonical injection $i_H:H\rightarrow H+H$ of the coproduct of $H$ with itself, it is a retraction. In other words, since $f$ is a weak equivalence, using  Corollary \ref{mapCore} to obtain a map $a:H\rightarrow G$, we can construct the following commutative diagram, which by the lifting property of acyclic cofibrations implies that $f$ is a retraction.
\[\xymatrix{
H \ar[r]^a\ar[d]_{i_1} & G \ar[d]^{f} \\
H+H \ar[r]_-{f\circ a+1_H} \ar@{-->}[ur] & H 
}\]

\noindent
{\bf Step 3.} Any map $f:G\rightarrow H$ factorizes as 
\[\xymatrix{
G \ar[r]_-{i_1}\ar@/^1pc/[rr]^{f} & G+H \ar[r]_-{f+1_H} & H
}\]
where the morphism $i_1$ is a canonical injection in a coproduct and $f+1_H$ a retraction. This finishes the proof that the cofibrations and the acyclic fibrations form a weak factorization system.

We now prove that any $f:G\rightarrow H$ factorizes in an acyclic cofibration and a fibration.
 We have
\[\xymatrix{
G \ar[r]_-{i_1}\ar@/^1pc/[rr]^{f} & G+G\times H \ar[r]_-{f+p_2} & H
},\]
where $i_1$ is the canonical injection and $p_2$ is projection of the product to its second factor.  
The map $i_1$ is an acyclic cofibration. It is a section because $G\times H$ maps to $G$ and it is clearly a canonical injection in a coproduct. The map $f+p_2$ is a fibration because it has the appropriate lifting property with respect to acyclic cofibrations. We check this by taking a commutative diagram of solid arrows:
 \[\xymatrix{
A \ar[r]_-{a}\ar[d] & G+G\times H \ar[d]^{f+p_2} \\
A+B \ar[r]_{b} \ar@{-->}[ur]^{h} & H 
}\]
such that there is a map $c:B\rightarrow A$ and constructing $h$. We set $h=i_1\circ a+i_2\circ((a \circ c)\times b)$, where $i_1$ and $i_2$ are the canonical injections in $G+G\times H$.

\noindent
{\bf Step 4.} It only remains to prove that the Two-of-three axiom holds. We consider the diagram $\xymatrix{
F \ar[r]_{f}\ar@/^1pc/[rr]^{f\circ g} & G \ar[r]_g & H}.$
 If two of the three maps $f$, $g$ and $f\circ g$ are weak equivalences,  by Corollary \ref{isoCore}, the three graphs $F,G$ and $H$ have isomorphic cores. Again by Corollary \ref{isoCore}, this implies that the three maps are weak equivalences.
 
\end{proof}
\begin{remark}
There is a strengthening of the definition of model structures, which asks for the two factorizations to be functorial. The core model structure satisfies this stronger definition, since the constructions used in the proof of factorization are completely functorial. We will in further work \cite{droz} give a more general version of this model structure for all finitely complete and finitely cocomplete categories.
\end{remark}
We now prove two basic properties of the model category obtained.

\begin{theorem}
Every object is fibrant and cofibrant in the core model structure.
\label{fibCof}
\end{theorem}
\begin{proof}
Since all morphisms from the empty graph are cofibrations, all graphs are cofibrant. To prove that all objects are fibrant, we need to show that a morphism $g$ from a graph $G$ to the terminal object is a fibration. As in the proof of the previous theorem, this follows from Lemma \ref{initialTerminalLifting}.
\end{proof}

\begin{theorem}
Any two morphisms with same domain and codomain are homotopic in the core model structure. 
\end{theorem}
\begin{proof}
Since the coproduct of an object with itself is a very good cylinder object, any two morphisms are left homotopic. Because of Theorem \ref{fibCof}, we do not need to discriminate between left and right homtopies. See \cite{dwyerSpalinsky} for explanations about the general definition of homotopy between morphisms in term of cylinder objects.
\end{proof}

\begin{corollary}
The homotopy category of the core model structure is $\mathcal{G}_{core}$.
\end{corollary}

\section{The number of model structures}
\label{tnomcs}
After examples of model structures on the category of graphs are found, it becomes natural to try to classify the model structures. In this section, we attain the modest goal of counting them.

 Since the number of isomorphism types of finite graphs is countable and the number of morphisms between finite graphs is finite, the total number of morphisms of the category $\mathcal{G}$ is countable. Since model structures are described by the choice of three subsets of morphisms, there are at most $2^{\aleph_0}$ model structures on $\mathcal{G}$. We show that there are in fact exactly $2^{\aleph_0}$ model structures on $\mathcal{G}$. This might indicate that classifying the model structures is difficult.

\begin{definition}
In a quasi-order $(\mathcal{Q},\geq)$, a {\it downward closed set} is a set $S\subset \mathcal{Q}$ such that $\forall x\in S (\forall y (x\geq y \Rightarrow y\in S))$. 
\end{definition}

\begin{definition}
\label{chromatic}
The {\it chromatic number} of a graph $G$ is the smallest number $k$ such that if the vertices of $G$ are colored with less than $k$ colors, two adjacent vertices have the same color. We say that a graph is {\it $n$-colorable} if its chromatic number is smaller or equal to $n$.
\end{definition}

\begin{theorem}
There is a continuum of model structures on $\mathcal{G}$. Moreover, they give different homotopy categories.
\end{theorem}
\begin{proof}
\label{continuumOfStructures}
We construct a family of model structures parametrized by downward closed subsets of $Ob(\mathcal{G})$ considered as a quasi-order (see Definition \ref{quasiOrder}). The theorem is implied by Lemmata \ref{manyIncomp} and \ref{collapsedCat}. Lemma \ref{manyIncomp} states that there are $2^{\aleph_0}$ downward closed subsets of $Ob(\mathcal{G})$ and Lemma \ref{collapsedCat} says that our construction gives different homotopy categories for different downward closed subsets. To make Lemma \ref{collapsedCat} simple to prove, we will assume that the downward closed sets do not contain graphs of chromatic number smaller than $3$.
\\
We use $\mathcal{K}$ to denote both a downward closed subset of  $Ob(\mathcal{G})$ and the subcategory of $\mathcal{G}$ it induces.
We define the model structure $\mathbb{M}(\mathcal{K})$.
\begin{itemize}
\item{As cofibrations $\mathcal{M}_{cof}(\mathcal{K})$, we simply take all morphisms of $\mathcal{G}$.}
\item{As weak equivalences $\mathcal{M}_{we}(\mathcal{K})$, we take all the isomorphisms and all the morphisms of $\mathcal{K}$.}
\item{As fibrations $\mathcal{M}_{fib}(\mathcal{K})$, we take all the morphisms that lift on the right of all the acyclic cofibrations.}
\end{itemize}

We check the axioms in Definition \ref{originalAxioms}, beginning with the Retraction axiom for weak equivalences. If $h:G'\rightarrow H'$ is the retraction of a weak equivalence $f:G\rightarrow H$, we have maps $G\rightarrow G'$ and $H\rightarrow H'$. Therefore, since $G,H \in \mathcal{K}$ and $\mathcal{K}$ is downward closed, $G',H' \in \mathcal{K}$. We deduce that $h$ is also a weak equivalence. This proves the Retraction axiom for weak equivalences.

 The set of fibrations contains exactly the maps with domain not in $\mathcal{K}$ and the isomorphisms. It contains the maps with domains outside of $\mathcal{K}$ because of the non-existence of commutative square that could contradict the lifting property. The set of fibrations does not contain any map in $\mathcal{K}$ except the isomorphisms, because those maps do not lift on the right of themselves. This implies that the set of acyclic cofibrations is maximal for lifting on the left of the fibrations. We can apply Theorem \ref{riehl} to deduce the Retraction axiom for fibrations. We also deduce that the acyclic fibrations are the isomophisms. 

The Retraction axiom, Lifting axiom and the Factorization axiom for cofibrations and acyclic fibrations are clearly satisfied. The set of fibrations is by construction the maximal set of maps lifting on the right of the acyclic cofibrations. Thus the Lifting axiom is verified.

The only remaining axioms of model structures requiring a proof are the Two-of-three axiom and factorization of any map in an acyclic cofibration and a fibration. Let $f:G\rightarrow H$ be any morphism of $\mathcal{G}$, if $G$ and $H$ are in $\mathcal{K}$, then $f$ is an acyclic cofibration and there is nothing to prove. If $G\notin{\mathcal{K}}$, as we observed above, $f$ is a fibration and again, $f$ factorizes trivially.
The Two-of-three axiom follows easily from the fact that if two of the maps $f,g,f\circ g$ are weak equivalences, either their three domains and ranges must be in $\mathcal{K}$, or all three must be isomorphisms. 
\end{proof}

We observe that a graph has a cycle of odd length exactly if it is not $2$-colorable.
\begin{definition}
\label{oddgirth}
The {\it girth} of a graph with at least one cycle is the size of its smallest cycle. The {\it odd girth} of a graph with chromatic number at least $3$ is the size of its smallest cycle of odd length.
\end{definition}

\begin{lemma}
\label{manyIncomp}
There are $2^{\aleph_0}$ different downward closed subsets of $Ob(\mathcal{G})$ that only contain graphs of chromatic number greater than $2$.  
\end{lemma}
\begin{proof}
We first show that we can construct a countable set $S$ of graphs without homomorphisms between them as in \cite[Theorem 3.10]{nesetril}. 
Let $S_0=\emptyset$. We construct $S_{n+1}$ by adding to $S_n$ a graph with the following properties: $G$ is a graph with girth and chromatic number both greater than the odd girth or chromatic number of any graph in $S$ and containing an odd cycle (This last condition is implied by the chromatic number being greater than $2$). Such a graph exists, by the classic result of Erd\"os \cite{erdos} that there are graphs with both arbitrary high chromatic number and arbitrary high girth. The chromatic number of $G$ is an obstruction to the existence of a homomorphism from $G$ to any graph of $S_n$, and the odd girth is an obstruction to the existence of a homomorphism from a graph in $S_n$ to $G$. We set $S=\bigcup_{i\in\mathbb{N}} S_i$, $S$ is by construction a countable set of incomparable graphs.
For any of the $2^{\aleph_0}$ subsets $U\subset S$, $\{G\mid \exists H \in U \,\mathrm{with}\, H\rightarrow G\}$ is a downward closed set. Those downward closed sets are all different because they have different sets of maximal elements.  
\end{proof}
\begin{definition}
For a downward closed subset $\mathcal{K}$ of $\mathcal{G}$ considered as a quasi-order, we write $\mathcal{G}_{\mathcal{K}}$ for the full subcategory of $\mathcal{G}$ induced by the graphs that are not in $\mathcal{K}$ and the terminal object of $\mathcal{G}$. We write $F$ for the functor $F_\mathcal{K}:\mathcal{G}\rightarrow\mathcal{G}_{\mathcal{K}}$ that sends every object of $\mathcal{G}$ that is not in $\mathcal{K}$ to itself and all other objects to the terminal object of $\mathcal{G}_{\mathcal{K}}$ and sends a morphism $f$ of $\mathcal{G}$ either to the same homomorphism between graphs in $\mathcal{G}_{\mathcal{K}}$ or to the unique morphism between the image of the domain of $f$ and the terminal object of $\mathcal{G}_{\mathcal{K}}$.   
\end{definition} 
\begin{lemma}
\label{collapsedCat}
The localization of the category $\mathcal{G}$ at a subcategory induced by a downward closed set $\mathcal{K}$ is equivalent to $\mathcal{G}_{\mathcal{K}}$. Moreover, if $\mathcal{K}\neq \mathcal{L}$ are two downward closed subsets that do not contain graphs of chromatic number smaller than $3$, the categories $\mathcal{G}_{\mathcal{K}}$ and $\mathcal{G}_{\mathcal{L}}$ are not equivalent.
\end{lemma}
\begin{proof}
Let $\mathcal{K}^{-1}\mathcal{G}$ be the localization of $\mathcal{G}$ at the maps of $\mathcal{K}$. By the universal property of the localization, the functor $F_\mathcal{K}$ factorizes through $\mathcal{K}^{-1}\mathcal{G}$. By the basic properties of the localization, the factorization provides the promised equivalence of categories. 
\[\xymatrix{
 & K^{-1}\mathcal{G}\ar@{->}[rd] & \\
\mathcal{G}\ar[rr]\ar[ru]& &\mathcal{G}_{\mathcal{K}}
}\]
The non-equivalence of $\mathcal{G}_{\mathcal{K}}$ and $\mathcal{G}_{\mathcal{L}}$ is proven by an argument similar to the proof of Theorem \ref{noAutomorphism} about the triviality of automorphisms of $\mathcal{G}$. The restriction on the chromatic number ensures that the small graphs appearing in the proof of Theorem \ref{noAutomorphism} are available.
\end{proof}

\end{document}